\documentclass[12pt]{amsart}
\usepackage{amsmath, amssymb, mathtools, a4wide}
\usepackage{paralist}
\usepackage[utf8]{inputenc}
\usepackage[T1]{fontenc}
\usepackage{lineno}
\showboxdepth=\maxdimen
\showboxbreadth=\maxdimen
\numberwithin{equation}{section}

\theoremstyle{plain}
\newtheorem{thm}{Theorem}

\newtheorem*{prop*}{Proposition}
\newtheorem{lem}{Lemma}[section]
\newtheorem{prop}[lem]{Proposition}

\newcommand{\thmref}[1]{Theorem~\ref{#1}}

\theoremstyle{definition}
\newtheorem{rmk}[lem]{Remark}
\newtheorem{defi}{Definition}

\newcommand{\mbf}{\mathbf}

\newcommand{\mf}{\mathbf}
\newcommand{\q}{\quad}
\newcommand{\qq}{\qquad}

\newcommand{\mc}{\mathcal}

\newcommand{\mrm}{\mathrm}

\newcommand{\sltwo}{\mrm{SL}_2(\mf Z)}

\newcommand{\slz}{\mrm{SL}_2(\mbf Z)}

\begin{document}
\title[Hecke-Siegel type threshold]{Hecke-Siegel type threshold for square-free Fourier coefficients: an improvement.}

\author{Pramath Anamby}
\address{Department of Mathematics\\
Indian Institute of Science\\
Bangalore -- 560012, India.}
\email{pramatha@iisc.ac.in, pramath.anamby@gmail.com}

\author{Soumya Das}
\address{Department of Mathematics\\
Indian Institute of Science\\
Bangalore -- 560012, India.}
\email{soumya@iisc.ac.in, soumya.u2k@gmail.com}


\begin{abstract}
We prove that if $f$ is a non zero cusp form of weight $k$ on $\Gamma_0(N)$ with character $\chi$ such that $N/(\text{conductor }\chi)$ square-free, then there exists a square-free $n\ll_{\epsilon} k^{3+\epsilon}N^{7/2+\epsilon}$ such that $a(f,n)\neq 0$. This significantly improves the already known existential and quantitative  result from previous works.
\end{abstract}
\maketitle

\section{Introduction}

It is an obvious fact that the Fourier coefficients of a modular form completely determine it. Thus an interesting question to ask is, whether any subset of set of all Fourier coefficients determines the form. The oldest result in this direction says that if $f$ is a form in $M_k(\Gamma)$ (here $\Gamma \subseteq \sltwo$ is a congruence subgroup, $k\geq 1$, and $M_k(\Gamma)$ is the space of modular forms of weight $k$ with respect to $\Gamma$) with a Fourier expansion, say for $\tau \in \mc H=\{ z \in \mf C \mid \Im(z)>0\}$, 
\begin{align} \label{fou}
 f(\tau) = \sum_{n=0}^\infty a(f,n) e^{2 \pi i n \tau}, 
\end{align} 
then there exists a number $A \geq 0$ depending on the space such that if $a(f,n)=0$ for all $n \leq A$, then $f=0$. The smallest such bound or threshold is due to Hecke \cite{hecke} and is known to be of the same magnitude as the quantity $ k \cdot  [\slz \colon \Gamma]$, a similar result is true for Siegel modular forms by results due to Siegel (see \cite{efrie} for example). This bound is popularly known as (mistakenly though) Sturm's bound, who proved a ‘mod $p$’ version of this result.

Many other affirmative results are available in the literature like the "multiplicity-one" results in the case of elliptic newforms of integral weights, and are crucial in many applications, both analytic and arithmetic in nature. For example, the same question in the setting of half-integral weight modular forms has a bearing to questions like non-vanishing of central $L$-values \cite{luo-rama}, and that for Siegel modular forms to certain automorphic lifting theorems \cite{saha} etc. We refer the reader to \cite{pra-sou1} for more extensive discussion on this.

Let us now discuss some of the recent results in this line of investigation. In \cite{pra-sou1}, the authors proved (essentially) that the elliptic cusp forms of integral weights and square-free levels are determined by their `square-free' Fourier coefficients (i.e., by those which are indexed by square-free numbers). This was motivated by the quest of such a result for the so-called Hermitian modular forms, which are automorphic with respect to the unitary group $U(n,n)$ over the ring of integers of an imaginary quadratic field. Let us recall those results in some detail.

Let $N$ be a positive integer and $\chi$ a Dirichlet character mod $N$ with conductor $m_\chi$. When $N/m_\chi$ is square-free it was proved in \cite{pra-sou1} that the set of all square-free Fourier coefficients determine any cusp form in $M_k(N,\chi)$ (see section~\ref{sec:setup}). Further in \cite{pra-sau2}, the existence of an analogue of the Sturm's bound for the square-free Fourier coefficients was proved. To be more precise, let us define $\mu_{\textsf{sf}}(k,N)$ to be the smallest integer such that whenever $f \in S_k(N,\chi)$ and $a(f,n)=0$ for all \textit{square-free} $n\le \mu_{\textsf{sf}}(k,N)$, then $f=0$. It is not a-priori clear that such a bound should exist. This was indeed shown to exist in \cite{pra-sau2}, and the following rather crude bound was shown. In particular the bound is exponential in the weight and level.
\begin{equation}\label{prbound}
\mu_{\textsf{sf}}(k,N)\le a_0\cdot N\cdot 2^{\frac{r(r-1)}{2}}e^{4r\log^2(7k^2N)},
\end{equation}
where $a_0$ is an absolute constant and $r=(k-1)N$.

The idea in \cite{pra-sau2} was to reduce to newforms following an argument of Balog and Ono \cite{BalogOno}, where one needs to work with many primes at which two distinct newforms have distinct eigenvalues. One way to handle this is the prime number theorem (PNT) for newforms, and the bad bound is due to the error term in the PNT.

The purpose of this article is to improve the above bound vastly. The main idea is to work with suitably modified $L$-functions, and replace the PNT by the Rankin-Selberg method. Of course one has to keep track on the dependence of the `analytic conductor' (essentially a function of weight and level); and reduce oneself to the case of newforms. The latter step is a little tricky. Let us now state the main result of this paper.

\begin{thm}\label{th:main}
Let $N$ be a positive integer and $\chi$ a Dirichlet character mod $N$ with conductor $m_\chi$ such that $N/m_\chi$ is square-free. Let $f\in S_k(N,\chi)$ be non-zero and fix any $\epsilon>0$. Then there exists a square-free integer $n\ll k^{3+\epsilon}N^{7/2+\epsilon}$ such that $a_f(n)\neq 0$, the implied constant depending only on $\epsilon$.
\end{thm}
Clearly this is a significant improvement over the previously known bound. This was in part motivated by an asymptotic in \cite{ko-ro-wu} on a similar subject. But we follow a simpler approach to get our result. By exploiting the properties of Rankin-Selberg $L$-functions, first we obtain either an asymptotic, or an upper bound for a (suitable smooth) weighted sum of the products $\lambda_f(n) \lambda_g(n)$ of square-free Fourier coefficients of normalized Hecke newforms, with explicit error. Once these have been established, the proof of \thmref{th:main} is obtained by reducing to this case from an arbitrary cusp form using newform theory.

\noindent\textbf{Acknowledgements.} This work was supported by the Research Institute for Mathematical Sciences, a Joint Usage/Research Center located in Kyoto University. It is a great pleasure for the second author to acknowledge the financial support and an enriching mathematical experience at the RIMS conference ``Analytic and Arithmetic Theory of Automorphic Forms'', held in January 2018. The first author is a DST- INSPIRE fellow at IISc, Bangalore and acknowledges the financial support from DST (India). The second author acknowledges financial support in parts from the UGC Centre for Advanced Studies, DST (India) and IISc, Bangalore during the completion of this work.

\section{Setup}\label{sec:setup}
\subsubsection*{General notation and preliminaries} Let $N$ be a positive integer and $\chi$ a Dirichlet character mod $N$ with conductor $m_\chi$. Then $M_k(N,\chi)$ denotes the space of modular forms of weight $k$ on $\Gamma_0(N)$ with character $\chi$ and $S_k(N,\chi)\subset M_k(N,\chi)$ denotes the space of all cusp forms.

We use the usual $\epsilon$ convention in analytic number theory: $\epsilon>0$ is an arbitrarily small number which may vary at different occurances. Moreover we adopt the standard Landau $O$-symbol: $A \ll B$ means $A \leq (\mrm{constant}) \cdot B$ with the constant depending on certain parameters at hand, usually mentioned explicily.
\begin{defi}["Naive" Rankin-Selberg Convolution]
Let $f,g\in S_k(N,\chi)$ be normalized Hecke newforms for some level dividing $N$.  Let $\lambda_f(n)$ and $\lambda_g(n)$ be the Fourier coefficients of $f$ and $g$ respectively. Then the "naive" Rankin-Selberg convolution is defined as
\begin{equation}
L(f\times \overline{g},s):=\sum_{n\ge 1} \lambda_f(n)\overline{\lambda_g(n)}n^{-s}, \q\qq \text{ for }\; \mrm{Re}(s)>1.
\end{equation}
\end{defi}
The usual Rankin-Selberg convolution $L(f\otimes g,s)$ of $f$ and $g$ is defined as
\begin{equation}\label{eq:relation}
L(f\otimes \overline{g},s):=\prod_{p}\prod_{i,j=1}^{2}(1-\alpha_i(p)\overline{\beta_j(p)}p^{-s})^{-1},
\end{equation}
where for a prime $p$, the Satake-parameters $\alpha_1(p),\alpha_2(p)$ and $\beta_1(p),\beta_2(p)$ are the roots of the quadratic polynomials $x^2-\lambda_f(p)x+\chi(p)$ and $x^2-\lambda_g(p)x+\chi(p)$ respectively. 

Since $\lambda_f(n)$ and $\lambda_g(n)$ are multiplicative, $L(f\times \overline{g},s)$ has an Euler product. Let $L_p(f\times \overline{g},s)$ denote the $p$th Euler factor of $L(f\times \overline{g},s)$. Then for $p\nmid N$, we have (see \cite[page 133]{iw-ko})
\begin{equation*}
L_p(f\times \overline{g},s)=(1-p^{-2s})\prod_{i,j=1}^{2}(1-\alpha_i(p)\overline{\beta_j(p)}p^{-s})^{-1}.
\end{equation*}
Let $L_N(f\times \overline{g},s):=\prod_{p\nmid N}L_p(f\times \overline{g},s)$. Then we can write
\begin{equation}\label{eq:LH}
L_N(f\times \overline{g},s)=L(f\otimes \overline{g},s) H(s),
\end{equation}
where $H(s)=\prod_p H_p(s)$ and $H_p(s)$ is given by
\begin{equation}\label{eq:hs}
H_p(s)=\begin{cases}
(1-p^{-2s}) \q\text{ if } p\nmid N;\\
\prod_{i,j=1}^{2}(1-\alpha_i(p)\overline{\beta_j(p)}p^{-s}) \text{ otherwise }.
\end{cases}
\end{equation}
Since $H(s)$ is equal to $\zeta^{(N)}(2s)^{-1}(:=\prod_{p\nmid N}(1-p^{-2s}))$ up to finitely many Euler products, $H(s)$ converges absolutely for $\mrm{Re}(s)>1/2$.

We would prefer to work with smooth cut-off functions from now on, and hence we consider a smooth and positive function $\omega$ with support in $[\tfrac{1}{2},1]$. The Mellin transform of $\omega$ is given by 
\begin{equation*}
\widetilde{\omega}(s):=\int_{0}^{\infty}y^{s-1}\omega(y)dy.
\end{equation*}
The integral converges for any $s\in\mbf C$, thus $\widetilde{\omega}(s)$ is entire and since $\omega$ is smooth and compactly supported, using integration by parts we get 
\begin{equation}\label{eq:omegadecay}
\widetilde{\omega}(s)\ll |s|^{-A-1}
\end{equation}
for any $A>0$, the implied constant depends only on $A$ and $\omega$.

Let $\sideset{}{^\#}\sum$ denote the sum over square-free integers. Then we prove the following result.
\begin{prop}\label{prop:asymp}
Let $f, g\in S_k(N,\chi)$ be normalized newforms for some level dividing $N$. Then we have for any $1/2< c< 1$ and $\epsilon>0$, the following
\begin{itemize}
\item [(i)] There exists a constant $C(f,\omega)>0$ such that
\[\underset{(n,N)=1}{\sideset{}{^\#}\sum_{n\ge 1}}|\lambda_f(n)|^2\omega(n/x)=C(f,\omega) x+O(x^{c}k^{1-c+\epsilon}N^{\frac{3(1-c)}{2}+\epsilon}).\]
\item [(ii)] If $f\neq \overline{g}$ \[ \underset{(n,N)=1}{\sideset{}{^\#}\sum_{n\le x}}\lambda_f(n)\overline{\lambda_g(n)}\omega(n/x)=O(x^{c}k^{1-c+\epsilon}N^{\frac{3(1-c)}{2}+\epsilon}).\]
\end{itemize}
In both (i) and (ii), the implied constants depend only on $\epsilon>0$.  Moreover, $ C(f,\omega) \gg_\epsilon  (kN)^{-\epsilon}$.
\end{prop}
\begin{proof}
Let
\begin{equation}
L^\flat (f\times \overline{g},s):=\underset{(n,N)=1}{\sideset{}{^\#}\sum_{n\ge 1}}\lambda_f(n)\overline{\lambda_g(n)}n^{-s}=\prod_{p\nmid N}(1+\lambda_f(p)\overline{\lambda_g(p)}p^{-s}).
\end{equation}
Then we can write 
\begin{equation}\label{eq:LNH1}
L^\flat (f\times \overline{g},s)=L_N(f\times \overline{g},s)H_1(s),
\end{equation}
where $H_1(s)=\prod_{p\nmid N} H_{1,p}(s)$ and $H_{1,p}(s)$ is given by
\begin{equation*}
H_{1,p}(s)=
(1+\lambda_f(p)\overline{\lambda_g(p)}p^{-s})L_p(f\times \overline{g},s)^{-1}.
\end{equation*}
Let $H_{1,p}(X)=(1+\lambda_f(p)\overline{\lambda_g(p)}X)(1-X^2)^{-1}\prod_{i,j=1}^{2}(1-\alpha_i(p)\overline{\beta_j(p)}X)$. Now noting that for these primes, $H_{1,p}^{'}(0)=0$, we get that $H_1(s)$ converges absolutely for $\mrm{Re}(s)>1/2$ (see \cite{sou-koh-sen} for similar arguments). Thus using (\ref{eq:LH}) and (\ref{eq:LNH1}) we get
\begin{equation}\label{eq:lb}
L^\flat (f\times \overline{g},s)=L(f\otimes \overline{g},s)H(s)H_1(s).
\end{equation} 
In the following calculations we make use of the uniform convexity bound for $L(f\otimes\overline{g},s)$. For $1/2\le \sigma\le 1$, this is given by (see \cite[Theorem 5.41]{iw-ko})
\begin{equation}\label{eq:convf=g}
L(f\otimes\overline{g},s)\ll \mbf q(f\otimes \overline{g},s)^{\frac{(1-\sigma)}{2}+\epsilon},
\end{equation}
where the implied constants depend only on $\epsilon$. Here $\mbf q(f\otimes \overline{g},s)$ denotes the analytic conductor of $L(f\otimes\overline{g},s)$ (see \cite[chapter 5]{iw-ko} for details).

Using the Mellin inversion formula for $\widetilde{\omega}(s)$ (see \cite[page 90]{iw-ko}) , we have
\begin{align}\label{eq:intrep}
\underset{(n,N)=1}{\sideset{}{^\#}\sum_{n\ge 1}}\lambda_f(n)\overline{\lambda_g(n)}\omega(n/x)&=\frac{1}{2\pi i}\int_{(2)} L^\flat (f\times \overline{g},s)x^s\widetilde{\omega}(s)ds\nonumber\\
&=\frac{1}{2\pi i}\int_{(2)} L(f\otimes \overline{g},s)H(s)H_1(s)x^s\widetilde{\omega}(s)ds.
\end{align} 
We use (\ref{eq:lb}) to get the previous equality. Now we move the line of integration to $1/2<c<1$ ($c$ will be chosen later). Since Rankin-Selberg convolution is polynomially bounded in vertical strips and $\widetilde{\omega}$ has a rapid decay given by (\ref{eq:omegadecay}), the horizontal integrals do not contribute. 

If $f\neq \overline{g}$, then $L(f\otimes \overline{g},s)$ is entire. Otherwise $L(f\otimes \overline{g},s)$ has a pole at $s=1$ (see \cite[page 97]{iw-ko}). Thus we have
\begin{equation}\label{int:feqg}
\underset{(n,N)=1}{\sideset{}{^\#}\sum_{n\ge 1}}\lambda_f(n)\overline{\lambda_g(n)}\omega(n/x)= \delta(f,\overline{g})\mrm{Res}_{s=1}(F(s))x+\frac{1}{2\pi i}\int_{(c)}L(f\otimes \overline{g},s)H(s)H_1(s)x^s\widetilde{\omega}(s)ds,
\end{equation}
where $\delta(f,\overline{g})=1$ if $f=\overline{g}$ and $0$ otherwise and $F(s)$ denotes the integrand in (\ref{eq:intrep}). We let
\begin{equation}
C(f,\omega):=\mrm{Res}_{s=1}(F(s))=H(1)H_1(1)\mrm{Res}_{s=1}L(f\otimes\overline{f},s)\widetilde{\omega}(1)
\end{equation}
Since $\omega$ and $H(1)H_1(1)$ are positive, $C(f,\omega)>0$. Also $\mrm{Res}_{s=1}L(f\otimes\overline{f},s)\gg (kN)^{-\epsilon}$ (for a non--CM form this can be improved to $\log (kN)^{-1}$, see \cite{HL}). Moreover, from the fact that $H_1(s)$ converges absolutely for $Re(s)>1/2$, it follows that $H(1)\gg 1$, with implied constant absolute. Further, from (\ref{eq:hs}), we easily see that 
\begin{equation*}
H(1)\gg \prod_{p|N}(1-1/p)^4\gg 2^{-4\nu(N)},
\end{equation*}
where $\nu(N)$ denotes the number of prime divisors of $N$. Invoking the standard bound $\nu(N)\ll \log N/\log\log N\ll_{\epsilon}N^\epsilon$, we finally get for any $\epsilon>0$ that 
\begin{equation}
C(f,\omega)\gg_\epsilon (kN)^{-\epsilon}.
\end{equation}
Now we estimate the integral on the line $c$. Since both $H(s)$ and $H_1(s)$ converge absolutely for $\mrm{Re}(s)>1/2$, we have $H(s)H_1(s)\ll_\epsilon 1$. Using the uniform convexity bound (\ref{eq:convf=g}) and that (see \cite[page 609]{har-mic})
\begin{equation*}
\mbf q(f\otimes \overline{g},s)\ll (1+|t|)^{4} k^2N^3.
\end{equation*}
the integral on the line $c$ is bounded by
\begin{equation*}
x^{c}k^{1-c+\epsilon}N^{\frac{3(1-c)}{2}+\epsilon}\int\limits_{0}^{\infty}(1+|t|)^{-A+1-2c+\epsilon}dt.
\end{equation*}
We choose $A=2-2c+2\epsilon$, so that the above integral converges absolutely. Thus we have
\begin{equation}
\underset{(n,N)=1}{\sideset{}{^\#}\sum_{n\ge 1}}\lambda_f(n)\overline{\lambda_g(n)}\omega(n/x)=\delta(f,\overline{g})C(f,\omega) x+O(x^{c}k^{1-c+\epsilon}N^{\frac{3(1-c)}{2}+\epsilon}).
\end{equation}
This completes the proof of the proposition.
\end{proof}

\section{Proof of \thmref{th:main}}
\begin{proof}
Let $\{f_{1}, f_{2},......f_{s}\}$ be a basis of newforms of weight $k$ and level dividing $N$ for $S_k(N,\chi)$. Now by the theory of newforms, for any non zero $f\in S_k(N,\chi)$, there exist $\alpha_{i,\delta} \in \mf C$ such that $f(\tau)$ can be written uniquely in the form
\begin{equation}\label{eq:decomp}
f(\tau)=\sum_{i=1}^{s} \sum_{\delta m_\chi |N} \alpha_{i,\delta}f_{i}(\delta \tau)
\end{equation}
such that at least one $\alpha_{i,\delta}\neq 0$. Note that in the above summation, for $\delta>1$, $\alpha_{i,\delta}=0$ if $f_i$ is not a newform of level $N/\delta$. Moreover, since $N/m_\chi$ is square-free, we have $\delta$ is square-free in the above summation. Let $d_{0}$ be the smallest divisor of $N$ such that $\alpha_{i,d_{0}}\neq 0$ for some $i$. 

For $(n,N)=1$ we have
\begin{equation*}
a_f(d_0n)=\sum_{i=1}^{s} \sum_{\delta m_\chi |N} \alpha_{i,\delta}\lambda_{f_i}(\tfrac{d_0n}{\delta}).
\end{equation*}
For $\delta< d_0$, $\alpha_{i,\delta}=0$ by our choice of $d_0$. Also, since $(n,N)=1$, $\lambda_{f_i}(\tfrac{d_0n}{\delta})=0$ whenever $\delta\neq d_0$. Thus, after renumbering if necessary, we can write for some $r\le s$
\begin{equation*}
a_f(d_0n)=\sum_{i=1}^{r}\alpha_{i,d_0}\lambda_{f_i}(n).
\end{equation*}
Now summing over all such square-free $n$ with the weight function $\omega$ we get
\begin{equation}\label{eq:sumnew}
\begin{split}
{\underset{(n,N)=1}{\sideset{}{^\#}\sum_{n\ge 1}}}|a_f(d_0n)|^2\omega\Big(\frac{n}{x}\Big)&={\underset{(n,N)=1}{\sideset{}{^\#}\sum_{n\ge 1}}}\sum_{i=1}^{r}|\alpha_{i,d_0}|^2|\lambda_{f_i}(n)|^2\omega\Big(\frac{n}{x}\Big)\\
&\q+{\underset{(n,N)=1}{\sideset{}{^\#}\sum_{n\ge 1}}}\underset{i\neq j}{\sum_{i,j=1}^{r}}\alpha_{i,d_0}\overline{\alpha_{j,d_0}}\lambda_{f_i}(n)\overline{\lambda_{f_j}(n)}\omega\Big(\frac{n}{x}\Big).
\end{split}
\end{equation}
Note that since $(n,N)=1$, if $d_0\neq 1$, then all the $f_i$s appearing in the above sum are newforms of level $N/d_0$. If $d_0=1$, then the $f_i$s appearing in the above sum can be newforms of any level dividing $N$. Thus using the proposition (\ref{prop:asymp}), the l.h.s of (\ref{eq:sumnew}) is  
\begin{align*}
&\ge\sum_{i=1}^{r}|\alpha_{i,d_0}|^2\left(C(f_i,\omega) x+O(x^{c}k^{1-c+\epsilon}(N/d_0)^{\frac{3(1-c)}{2}+\epsilon})\right)\\
&\qq-\underset{i\neq j}{\sum_{i,j=1}^{r}}|\alpha_{i,d_0}\overline{\alpha_{j,d_0}}|\left(O(x^{c}k^{1-c+\epsilon}(N/d_0)^{\frac{3(1-c)}{2}+\epsilon})\right)\\
&\ge \left(\sum_{i=1}^{r}|\alpha_{i,d_0}|^2C(f_i,\omega)\right) x-|\sum_{i=1}^{r}\alpha_{i,d_0}|^2\left(O(x^{c}k^{1-c+\epsilon}(N/d_0)^{\frac{3(1-c)}{2}+\epsilon})\right).\\
\end{align*}
Now using $C(f_i,\omega)\gg_\epsilon (kN/d_0)^{-\epsilon}$ and the Cauchy--Schwarz in equality in the second term, we get that l.h.s of (\ref{eq:sumnew}) is
\begin{equation}\label{eq:finalwt}
\ge\sum_{i=1}^{r}|\alpha_{i,d_0}|^2\left(C_1(\epsilon)(kN/d_0)^{-\epsilon}x-C_2(\epsilon)(kN/d_0)x^{c}k^{1-c+\epsilon}(N/d_0)^{\frac{3(1-c)}{2}+\epsilon}\right),
\end{equation}
where $C_1(\epsilon)$ and $C_2(\epsilon)$ are constants depending only on $\epsilon$. Here we also use the fact that $r\le\text{dim}(S_k^{new}(N/d_0,\chi))\approx kN/d_0$ (for eg., see \cite{martin}).

The above inequality holds true for any such weight function $\omega$ defined as in section \ref{sec:setup}. Also we can choose $0\le\omega\le 1$ so that
\begin{equation}
{\underset{(n,N)=1}{\sideset{}{^\#}\sum_{n\ge 1}}}|a_f(d_0n)|^2\omega\Big(\frac{n}{x}\Big)\le {\underset{(n,N)=1}{\sideset{}{^\#}\sum_{x/2< n < x}}}|a_f(d_0n)|^2.
\end{equation}
Thus from (\ref{eq:finalwt}) we have
\begin{equation*}
{\underset{(n,N)=1}{\sideset{}{^\#}\sum_{x/2< n < x}}}|a_f(d_0n)|^2\ge\sum_{i=1}^{r}|\alpha_{i,d_0}|^2\left(C_1(\epsilon)(kN/d_0)^{-\epsilon}x-C_2(\epsilon)x^{c}k^{2-c+\epsilon}(N/d_0)^{\frac{5-3c}{2}+\epsilon}\right).
\end{equation*}
We choose $c=1/2+\epsilon$ and the result by noting that the r.h.s is $>0$ for $x\gg_\epsilon k^{3+\epsilon}(N/d_0)^{\tfrac{7}{2}+\epsilon}$.
\end{proof}

\begin{rmk}
The use of a uniform sub-convexity bound for $L(f\otimes\overline{g}, \tfrac{1}{2}+\epsilon+it)$ instead of the convexity bound (\ref{eq:convf=g}) will reduce the exponent of $k$ and $N$ further by a small amount. For example, the use of sub-convexity result from \cite{Mic-aks} to get a sub-convexity bound for $L(f\otimes\overline{g}, \tfrac{1}{2}+\epsilon+it)$ and using this to bound the integral in (\ref{int:feqg}) will give us a slightly better exponents of $k$ and $N$.
\end{rmk}

\end{document}